\documentclass[12pt]{amsart}
\usepackage{pdfsync}
\usepackage{graphicx,xcolor}
\usepackage{amssymb}
\usepackage{amsmath}
\usepackage{amsthm}
\usepackage{enumerate}
\usepackage{subcaption}
\usepackage{esvect}
\usepackage[left=3cm,top=3.8cm,right=3cm]{geometry} 
\usepackage{ucs}
\usepackage{cite}
\usepackage{amsxtra}
\usepackage{epstopdf}
\usepackage{bbold}
\usepackage[utf8x]{inputenc}
\usepackage{verbatim}                                  
\usepackage[toc,page]{appendix}                        
\usepackage{pdfsync} 

\usepackage[pagebackref]{hyperref}
\hypersetup{
   colorlinks,
    linkcolor={red!60!black},
    citecolor={blue!60!black},
    urlcolor={blue!90!black}
}

\DeclareMathOperator{\rad}{rad}
\DeclareMathOperator{\dist}{dist}

\DeclareMathOperator{\diam}{diam}
\DeclareMathOperator{\conv}{conv}

\newtheorem{theorem}{Theorem}
\newtheorem{definition}[theorem]{Definition}
\newtheorem{proposition}[theorem]{Proposition}
\newtheorem{corollary}[theorem]{Corollary}
\newtheorem{observation}[theorem]{Observation}
\newtheorem{lemma}[theorem]{Lemma}
\newtheorem{example}[theorem]{Example}

\newtheorem{conjecture}[theorem]{Conjecture}

\newcommand{\R}{\mathbb{R}}

\def\N{{\mathbb N}}

\def\eps{{\varepsilon}}
\makeatletter
\def\moverlay{\mathpalette\mov@rlay}
\def\mov@rlay#1#2{\leavevmode\vtop{%
   \baselineskip\z@skip \lineskiplimit-\maxdimen
   \ialign{\hfil$\m@th#1##$\hfil\cr#2\crcr}}}
\newcommand{\charfusion}[3][\mathord]{
    #1{\ifx#1\mathop\vphantom{#2}\fi
        \mathpalette\mov@rlay{#2\cr#3}
      }
    \ifx#1\mathop\expandafter\displaylimits\fi}
\makeatother

\newcommand{\cupdot}{\charfusion[\mathbin]{\cup}{\cdot}}

\begin{document}
\thispagestyle{empty}

\title{A survey on Newhouse thickness, Fractal intersections and Patterns}

\author{Alexia Yavicoli}
\address{Department of Mathematics, the University of British Columbia. 1984 Mathematics Road, Vancouver BC V6T 1Z2, Canada}
\email{alexia.yavicoli@gmail.com, yavicoli@math.ubc.ca}

\begin{abstract}
In this article, we introduce a notion of size for sets called thickness that can be used to guarantee that two Cantor sets intersect (the Gap Lemma), and show a connection among Thickness, Schmidt Games and Patterns. We work mostly in the real line, but we also introduce the topic in higher dimensions.
\end{abstract}

\maketitle

\tableofcontents

\section{Newhouse's thickness}


In the 1970s, S.~Newhouse \cite{Newhouse,Newhouse2} defined \emph{thickness} on the real line. Thickness is a notion of size of a compact set, and Newhouse gave in his famous \emph{Gap Lemma} a simple condition involving thickness that ensures that two compact sets intersect. Since then, many mathematicians working on dynamical systems and fractal geometry were interested in this notion of size (e.g. \cite{Astels,HKY,MT21,ST20,Williams,Y20,Y21,FalconerYavicoli,y22,BP22}).

Before giving the definition of thickness and stating the Gap Lemma, we are going to see any compact set $C\subset\R$ as the result of sequentially ``poking holes'' starting with an interval (the convex hull of the compact set):

Let $C$ be a compact set in $\R$. We denote by $I_1$ the convex hull of $C$. There is a sequence $(G_n)_n$ (that might be finite) formed by disjoint bounded open intervals that are the path-connected components of $I_1 \setminus C$.

Since these intervals are disjoint and contained in the finite interval $I_1$, we can assume that they are ordered by non-increasing length (in fact, in case there are infinitely many of them we have $\lim_{n \to \infty}|G_n|=0$). If there are several intervals of the same length, we choose any ordering by non-increasing length. We can construct the compact set $C$ by removing these gaps in order (see Figure \ref{fig:compact}).


\begin{figure}
  \includegraphics[width=0.6\textwidth]{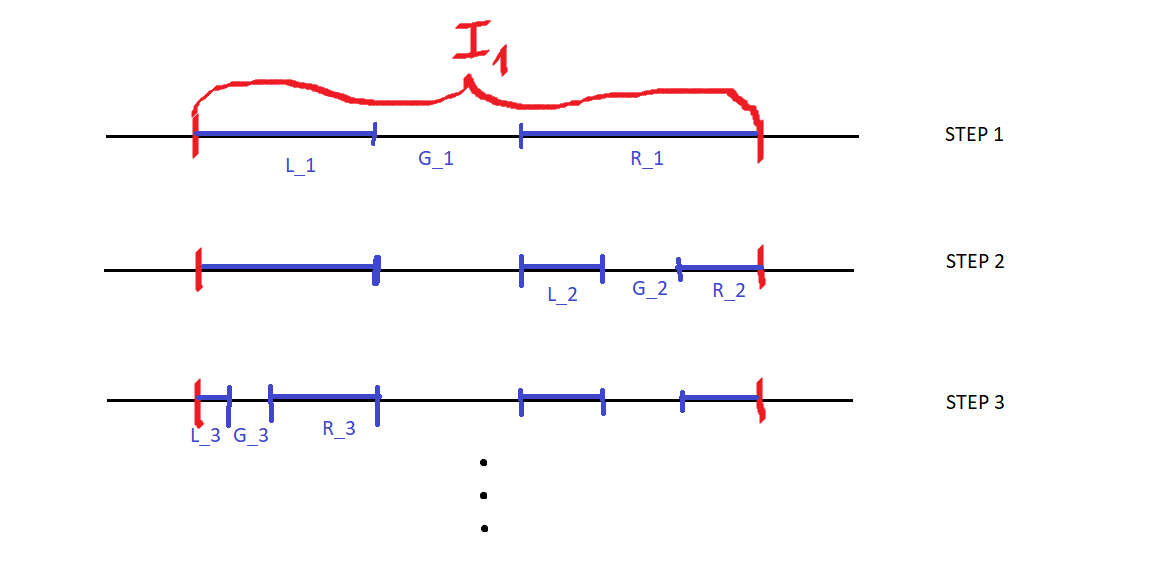}
\caption{Construction of a compact set}\label{fig:compact}
\end{figure}

When we remove $G_n$ from an interval of the previous step, we get two new intervals $L_n$ (at the left) and $R_n$ (at the right). Note that there may be degenerate intervals (singletons).

Thickness is a notion of size that looks at the smallest proportion of lengths of intervals over lengths of gaps:
\[
\tau (C):= \inf_{n} \frac{\min\{ |L_n|, |R_n| \}}{|G_n|}.
\]

\begin{observation}
When $\tau(C) \geq c$, then $\frac{|L_n|}{|G_n|}\geq c$ and $\frac{|R_n|}{|G_n|}\geq c$ for every $n$. Intuitively, this says that the set is large around each point of the set at every scale.
\end{observation}

\begin{observation}
In case $C$ has at least an isolated point $x$, then there exists $n$ so that $L_n=\{x\}$ or $R_n=\{x\}$, and so $\tau(C)=0$.
\end{observation}

\begin{lemma} Newhouse's thickness is well defined: any non-increasing order for the sequence of gaps gives the same value.
\end{lemma}
\begin{proof} One can prove that in case there are two gaps in the sequence with the same length $g:=|G_n|=|G_{n+1}|$, switching their order in the sequence gives the same thickness:
When the gaps are erased from different parents the quotients do not change, so the infimum does not change. In case the gaps $G_n$ and $G_{n+1}$ are erased from the same parent, the quotients may change but the infimum is the same when removing them in any order. Let's see the latter case:
\begin{figure}
  \includegraphics[width=0.6\textwidth]{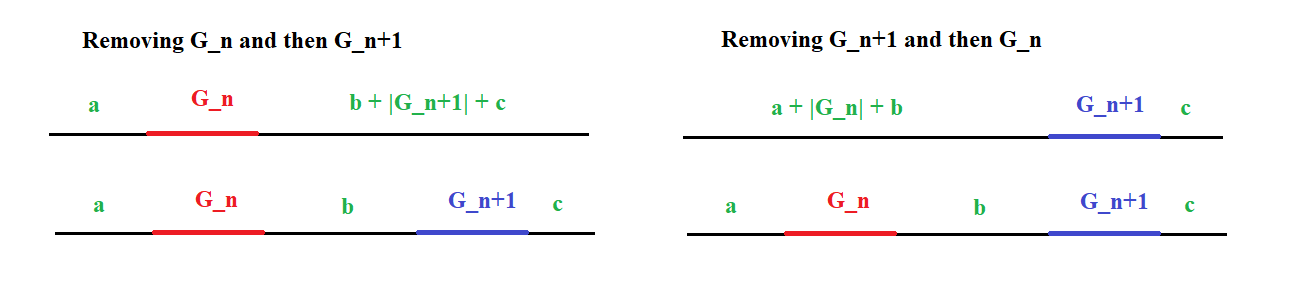}
\caption{Removing gaps of equal length in two possible orders}\label{fig:gaporder}
\end{figure}
When we remove $G_n$ and then $G_{n+1}$, see Figure \ref{fig:gaporder}, the quotients appearing in the definition of thickness are: $\frac{a}{g}, \frac{b+g+c}{g}, \frac{b}{g}, \frac{c}{g}$.
When we remove $G_{n+1}$ and then $G_n$ instead, see Figure \ref{fig:gaporder}, the quotients appearing in the definition of thickness are: $\frac{a+g+b}{g}, \frac{c}{g}, \frac{a}{g}, \frac{b}{g}$.
Note that \[\inf \{\frac{a}{g}, \frac{b+g+c}{g}, \frac{b}{g}, \frac{c}{g} \}=\inf\{ \frac{a}{g}, \frac{b}{g}, \frac{c}{g}\}=\inf \{ \frac{a+g+b}{g}, \frac{c}{g}, \frac{a}{g}, \frac{b}{g}\}.\]

Now, observe that since there are finitely many gaps with a fixed length, in finite steps one can order all gaps with the same length as $|G_n|$ through applying permutations as before.

After this, note that there is a sequence of steps $(N_k)_k$, where $|G_{N_k}|>|G_{N_k +1}|$, in which the thickness does not change when we reorder gaps with the same length up to those steps (we may reorder the first $N_k$ terms of the sequence, but sequence tails remain the same). We conclude that the thickness does not change.
\end{proof}

\begin{observation} In general the order of the sequence of gaps matters (i.e.: if we consider the sequence of gaps in an order that is not by non-increasing length, we may get a different result).
\end{observation}

\begin{observation}
The definition of thickness is invariant under homothetic functions: $\tau(aC+b)=\tau(C)$ for any $a \neq 0$, because homothetic functions preserve proportions. But, in general, thickness is \emph{not} invariant under smooth diffeomorphisms.
\end{observation}

\begin{example}[Thickness of central Cantor sets]\label{ex:thicknesscentral}
Given $\eps \in (0,1)$, let $M_{\eps}$ be the middle-$\eps$ Cantor set, which is obtained starting with the interval $[0,1]$ and then iterating the process of removing from each interval in the construction the middle open interval of relative length $\eps$ (see Figure \ref{fig:middle}).

\begin{figure}
  \includegraphics[width=0.6\textwidth]{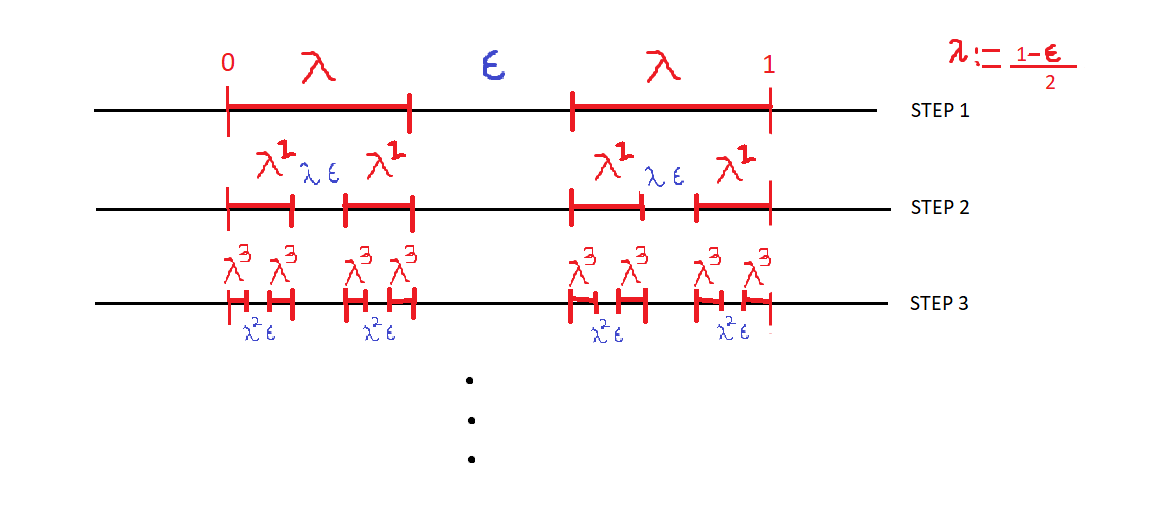}
  \caption{The middle-$\eps$ Cantor set}\label{fig:middle}
\end{figure}

Every time we remove a gap $G_n$ to get the step $m$ of the construction we have
\[\frac{\min\{|L_n|,|R_n|\}}{|G_n|}=\frac{\lambda^m}{\lambda^{m-1}\eps}=\frac{\lambda}{\eps}=\frac{1-\eps}{2\eps}.\]
Then, \[\tau(M_{\eps}):=\inf_n \frac{\min\{|L_n|,|R_n|\}}{|G_n|}=\frac{1-\eps}{2\eps}.\]

\end{example}

\section{The Gap Lemma}

\subsection{Why thickness and the Gap Lemma?}

Newhouse's motivation for defining thickness was the Gap Lemma, giving conditions for two compact sets in the real line to intersect. To motivate the definition of thickness and the assumptions of the Gap Lemma, we start by looking at the most basic non-trivial case:  two sets where each of them is formed by a union of two closed disjoint intervals.


If the compact sets are disjoint, then we have the following possible cases:
\begin{itemize}
\item Their convex hulls are disjoint.
\begin{center}\includegraphics[height=0.8cm]{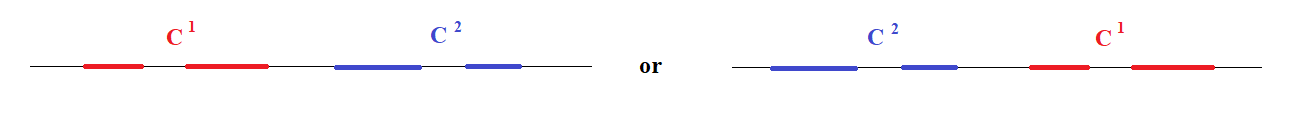}\end{center}
\item One of the sets is contained in a gap of the other set.
\begin{center}\includegraphics[height=0.8cm]{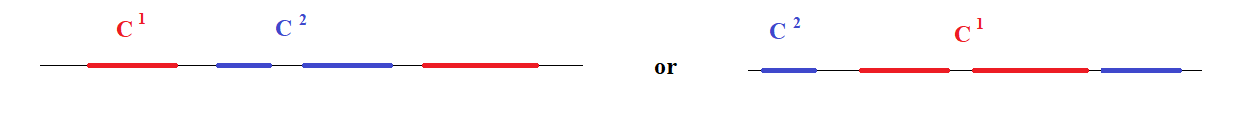}\end{center}
\item The sets are ``interleaved'', like this:
\begin{center}\includegraphics[height=1.4cm]{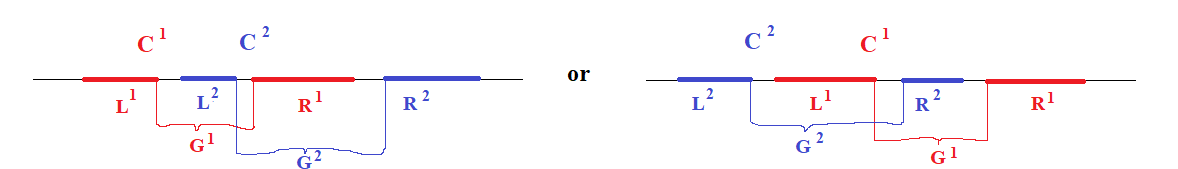}\end{center}
\end{itemize}

Let us study the first interleaved case: \begin{center}\includegraphics[height=1.4cm]{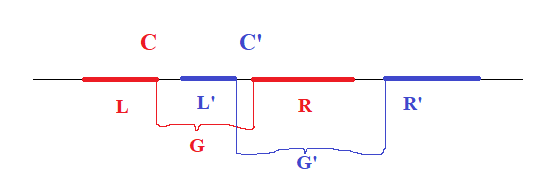}\end{center}

We have $|L^2|<|G^1|$ and $|R^1|<|G^2|$, hence  $\frac{|L^2| |R^1|}{|G^1| |G^2|} < 1$. So,
\[
\tau(C)\tau(C'):=\frac{\min \{|L^1|, |R^1| \}}{|G^1|} \frac{\min \{|L^2|, |R^2| \}}{|G^2|}< 1.
\]

The same happens with the other interleaved case. So, in this simple case we get:
\begin{lemma}[Baby Gap Lemma] \label{lem:baby-gap}
Let $C, C'$ be disjoint unions of two compact intervals.  If their convex hulls intersect, each set is not contained in a gap of the other one, and $\tau(C)\tau(C')\geq 1$, then $C\cap C'\neq\emptyset$.
\end{lemma}

\subsection{The Gap Lemma}

Newhouse's Gap Lemma (see for example \cite[Lemma 3.5]{Newhouse} for a first version of it, or in general \cite[Lemma 4]{Newhouse2}, see also \cite[page 63]{PalisTakens}) is a natural generalization of Lemma \ref{lem:baby-gap}, but now considering general compact sets in the line. We denote the convex hull of a set $C$ by $\conv(C)$.

\begin{theorem}[Newhouse's Gap Lemma]\label{NewhouseGapLemma}
Let $C^1$ and $C^2$ be two compact sets in the real line such that:
\begin{enumerate}
\item \label{GL1} $\conv (C^1) \cap \conv (C^2)\neq \emptyset$,
\item \label{GL2} neither set lies in a gap of the other compact set,
\item \label{GL3} $\tau(C^1) \tau(C^2)\geq 1$.
\end{enumerate}
Then, $$C^1 \cap C^2 \neq \emptyset.$$
\end{theorem}

Note that if we decide to consider the unbounded path connected components of the complement of the compact set as ``gaps'', then the second hypothesis would imply the first one. But in this survey we do not consider them as gaps, since their lengths do not appear in the denominator of the definition of thickness.

\begin{observation}[Sharpness of Theorem \ref{NewhouseGapLemma}]
Given two positive numbers $\tau_1, \tau_2$ so that $\tau_1 \tau_2<1$ we can construct compact sets $C^1, C^2$ with thickness $\tau_1, \tau_2$ respectively that are not contained in a gap of the other one and their intersection is empty:
Take
\[C^1:=[0,1]\cup [1+\tau_1^{-1}, 2+\tau_1^{-1}]\] which is a compact set with thickness $\tau_1$.
Since by hypothesis $\tau_1^{-1}>\tau_2>0$, there is an $\varepsilon \in (0, \frac{\tau_1^{-1}}{2})$ so that $\frac{\tau^{-1}-2\varepsilon}{1+2\varepsilon}=\tau_2$.
Define \[C^2=[-\tau_1^{-1}+\varepsilon, -\varepsilon]\cup [1+\varepsilon, 1+\tau_1^{-1}-\varepsilon]\] which is a compact set with thickness $\tau_2$.
It is easy to see that $C^1 \cap C^2 =\emptyset$, their convex hulls intersect, and none of them is contained in a gap of the other one.
\end{observation}

In order to prove the Gap Lemma we need  to introduce two definitions.

\begin{definition}
We say that two open intervals are \emph{linked} if each of them contains exactly one endpoint of the other (see Figure \ref{fig:linked}).

\begin{figure}
  \includegraphics[width=0.6\textwidth]{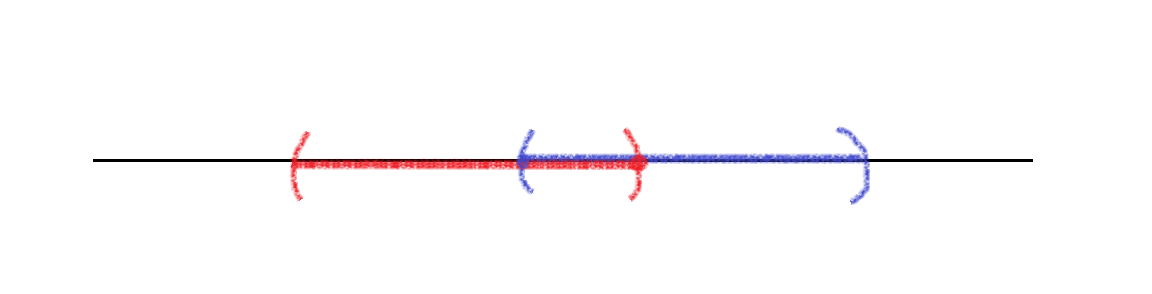}
  \caption{Linked intervals}\label{fig:linked}
\end{figure}
\end{definition}

\begin{definition}
Let $C$ be a compact set in the real line with sequence of gaps $(G_n)_n$ ordered by non-increasing length. Let $v$ be an endpoint of $G_n$.
We define the \emph{bridge} $B(v)$  associated to $v$ to be $L_n$ or $R_n$, depending on whether $v$ is the leftmost or rightmost point of $G_n$.
\end{definition}

\begin{proof}[Proof of Theorem \ref{NewhouseGapLemma}]

We are going to prove the Gap Lemma by contradiction. Let $C^1$ and $C^2$ be compact sets satisfying the assumptions of the Gap Lemma, and assume $C^1 \cap C^2 = \emptyset$.
Let $(G^1_m)_m$ and $(G^2_n)_n$ the sequences of gaps of the compact sets $C^1$ and $C^2$ ordered by non-increasing length.

It is enough to construct a sequence of pairs $(G_{m_i}^1, G_{n_i}^2)_i$ of linked gaps of $C^1$ and $C^2$ (where we advance in $n$ and $m$: $n_i<n_{i+1}$ and $m_i<m_{i+1}$). 
Then, taking $x_i$, $y_i$ to be the leftmost points of $G^1_{m_i}$, $G^2_{n_i}$, we get that  $x_i \in C^1$ and $y_i \in C^2$. Passing to subsequences if necessary, we may assume that $x_i\to x, y_i\to y$.
Observing that
\begin{align*}
\dist (x,y)&=\lim_{i \to \infty} \dist (x_i, y_i) \leq \lim_{i \to \infty} \diam (G^1_{m_i} \cup G^2_{n_i})\\
&\leq \lim_{i \to \infty} \diam (G^1_{m_i}) + \diam ( G^2_{n_i})=0,
\end{align*}
we get that $x=y\in C^1 \cap C^2$, which is a contradiction.

We are going to construct the sequence of pairs of linked gaps by induction. To begin, observe that:
\begin{itemize}
\item Any endpoint of the convex hull of $C^i$ or a gap of $C^i$ belongs to $C^i$.
\item If a point belongs to $\conv(C^i)$ but does not belong to $C^i$, then it is in a gap $G^i_m$.
\item If a point belongs to $C^i$, then (by the assumption $C^1 \cap C^2 = \emptyset$) the point is either outside of $\conv(C^j)$ ($j:=3-i$) or in a gap $G^j_n$.
\end{itemize}

In order to be able to handle the inductive step, we will prove a slightly stronger statement:
there is a sequence of pairs of gaps $G^1_{m_i}$ and $G^2_{n_i}$ that are linked, such that there is an endpoint of one so that its bridge is contained in the other gap.

\textbf{First step}.
By assumption \eqref{GL1} and symmetry, we may assume that there is an endpoint of $\conv(C^2)$ that belongs to $\conv(C^1)$, and so it is in $C^2 \cap \conv(C^1)$. But since $C^1 \cap C^2 =\emptyset$, then it belongs to $G^1_{m_1}\cap C^2$ for some gap $G^1_{m_1}$. So,
\begin{equation}\label{GLfirst} G^1_{m_1}\cap C^2 \neq \emptyset.\end{equation}

Since the endpoints of $G^1_{m_1}$ belong to $C^1$,  each of them must be either outside of $\conv(C^2)$ or in a gap of $C^2$. Since by construction $G^1_{m_1}$ contains an endpoint of $\conv(C^2)$, one of them is outside $\conv(C^2)$. The other one must belong to $\conv(C^2)$ (and therefore to a gap of $C^2$), because otherwise $C^2 \subseteq G^1_{m_1}$, contradicting the assumption \eqref{GL2}.

We have seen that there is an endpoint  of $G^1_{m_1}$ in a gap $G^2_{n_1}$, and there is an endpoint of  $G^1_{m_1}$ outside  $\conv (C^2)$. This implies that  there is an endpoint $v^2_{n_1}$ of $G^2_{n_1}$ so that $B(v^2_{n_1}) \subseteq G^1_{m_1}$ (in particular $G^1_{m_1}$ and $G^2_{n_1}$ are linked). This is the starting point of the induction.

\textbf{The inductive step}.
Assume that $G^1_{m_i}$ and $G^2_{n_i}$ are linked gaps, where $B(v^2_{n_i}) \subseteq G^1_{m_i}$ and $v^2_{n_i}$ is an endpoint of $G^2_{n_i}$ (the symmetric condition is identical). Let $u^1_{m_i}$ be the endpoint of $G^1_{m_i}$ that is in $G^2_{n_i}$.

Let us see that $B(u^1_{m_i})$ cannot be contained in $G^2_{n_i}$. Otherwise, since $B(u^1_{m_i}) \subseteq G^2_{n_i}$ and $B(v^2_{n_i})\subseteq G^1_{m_i}$, we would have
\[
\tau (C^1) \tau (C^2) \leq \frac{|B(u^1_{m_i})|}{|G^1_{m_i}|} \frac{|B(v^2_{n_i})|}{|G^2_{n_i}|}<1,
\]
which contradicts the thickness assumption \eqref{GL3}.

Then, the other endpoint $u^2_{n_i}$ of $G^2_{n_i}$ belongs to $B(u^1_{m_i})$, so $u^2_{n_i}\in C^2 \cap B(u^1_{m_i})\subseteq (C^1)^C \cap B(u^1_{m_i})$. So, $u^2_{n_i} \in G^1_{m_{i+1}}$ with $m_{i+1}>m_i$.

Then, taking $G^2_{n_{i+1}}:=G^2_{n_i}$ and $G^1_{m_{i+1}}$ as above, we got linked gaps with a bridge of an endpoint of $G^1_{m_{i+1}}$ contained in $G^2_{n_i}$ (because the bridge is contained between $G^1_{m_i}$ and $G^1_{m_{i+1}}$).

\end{proof}

\section{Connection to Hausdorff dimension}

Thickness is a notion of size of a set. Hausdorff dimension is another, more classical, notion of size. They are different, but related. In this section we explore the connections between these concepts.

We begin by recalling the definition of Hausdorff dimension (for a more complete background on Hausdorff dimension see \cite{FalconerFG,falconer_1985,FalconerTechniques}).

\begin{definition}
The $s$-dimensional Hausdorff content of $E$ is \[H_{\infty}^s(E):=\inf \left\{\sum_{k \in \N} \diam(U_k)^s : \ E\subseteq \bigcup_{k \in \N}U_k \text{ and } \{U_k\}_k \text{ is a family of open sets}\right\}.\]

\begin{center}  \includegraphics[width=4cm]{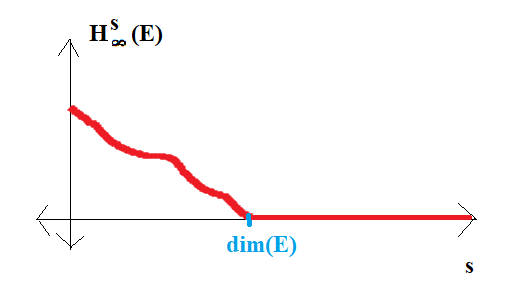} \end{center}

The \textbf{Hausdorff dimension} of $E$ is the supremum of all real-valued s for which the s-dimensional Hausdorff content of $E$ is positive.
\end{definition}

The following result shows that sets of large thickness also have large Hausdorff dimension:
\begin{theorem}\label{dimthickness}
Let $C \subset \R$ be a compact set with thickness $\tau (C)>0$, then
\[
\dim_H(C) \geq \frac{\log(2)}{\log(2+\frac{1}{\tau(C)})}.
\]
In particular, $\dim_H(C)\to 1$ as $\tau(C)\to\infty$.
\end{theorem}

\begin{observation}[Sharpness of Theorem \ref{dimthickness}]
Given $\tau \in (0, \infty)$, we can take $C:=M_{\eps}\subseteq \R$ to be the middle Cantor set with relative length of gaps $\eps:=\eps(\tau):=\frac{1}{2\tau+1}$. Then the relative length of each child in its parent is $\alpha:= \frac{\tau}{1+2\tau}$.
By Example \ref{ex:thicknesscentral} we have $\tau(C)=\frac{1-\eps}{2\eps}=\tau$, and by \cite[Example 4.4]{FalconerFG} (taking $m=2$) $\dim_H(C) = \frac{\log(2)}{\log(\alpha^{-1})}= \frac{\log(2)}{\log(2+\frac{1}{\tau})}$.\end{observation}

\begin{observation} Theorem \ref{dimthickness} shows that sets with large thickness have large Hausdorff dimension. The converse does not hold, in fact there are sets of positive Lebesgue measure (which implies full Hausdorff dimension) with thickness $0$: consider for example the union of a closed interval with an isolated point.
\end{observation}

\begin{proof}[Proof of Theorem \ref{dimthickness}]
We define
\[
\beta:= \frac{\log(2)}{\log(2+\frac{1}{\tau(C)})}.
\]

It is enough to prove that for any $\mathcal{U}:=\{U_i\}_i$ open covering of $C$ we have $\sum_i |U_i|^{\beta} \geq |\conv(C)|^\beta>0$.

Fix  $\mathcal{U}:=\{U_i\}_i$ an open covering of $C$.
Since $C \subseteq \R$ is a compact set, and the sum decreases while dropping elements of the sequence of $\mathcal{U}$, we can assume that the covering $\mathcal{U}$ is formed by finitely many elements.

The case in which $\mathcal{U}$ is formed by an interval $U$ is trivial (because $\conv(C) \subseteq U$ implies $|\conv(C)|^{\beta}\leq |U|^{\beta}$).
Let's assume that we have at least 2 open intervals in $\mathcal{U}$ and reduce inductively this case to the former case.

We define as before $(G_n)_n$ the sequence of gaps in the definition of $\tau (C)$, and $L_n$ and $R_n$ the left and right intervals associated to $G_n$.

Since $\mathcal{U}$ is a covering of $C$ by finitely many open sets, then the convex hull of $C$ except a few gaps are covered by $\mathcal{U}$, i.e. there exist $N \in \N$ and gaps $G_{m_1}, \cdots, G_{m_N}$ with $|G_{m_1}|\geq \cdots \geq |G_{m_N}|$ where \[\conv (C)\setminus \bigcup_{1\leq j\leq N} G_{m_j} \subseteq \bigcup_i U_i \text{ and } G_{m_j} \nsubseteq \bigcup_i U_i.\]
We have $L_{m_N}$ and $R_{m_N}$ the left and right intervals associated $G_{m_N}$ -which is the smallest gap that is not covered-.
Then, there exists $U(L)$ and $U(R)$ in $\mathcal{U}$ so that $L_{m_N} \subseteq U(L)$ and $R_{m_N} \subseteq U(R)$.

We define $A:=\conv (U(L)\cup U(R))$.
Using that $L_{m_N} \subseteq U(L)$, the definition of $\tau(C)$ and the definition of $A$ we get
\begin{equation}\label{eq:condR1}|U(L)|\geq |L_{m_N}| \geq \tau |G_{m_N}| \geq \tau (|A|-|U(L)|-|U(R)|).\end{equation}
Analogously, using that $R_{m_N} \subseteq U(R)$, the definition of $\tau(C)$ and the definition of $A$ we get
\begin{equation}\label{eq:condR2}|U(R)|\geq |R_{m_N}| \geq \tau |G_{m_N}| \geq \tau (|A|-|U(L)|-|U(R)|).\end{equation}

We define $x:=\frac{|U(L)|}{|A|}$ and $y:=\frac{|U(R)|}{|A|}$.
The proportions $x$ and $y$ satisfy:
\begin{itemize}
\item $x\geq 0$ and $y \geq 0$,
\item $x+y \leq 1$ (because $U(L) \cupdot U(R) \subseteq A$),
\item $\tau (1-x-y)\leq x$ and symmetrically $\tau (1-x-y)\leq y$ (because of \eqref{eq:condR1} and \eqref{eq:condR2}).
\end{itemize}

Based on this, we define the region $R$  by
\[
R:=\{(x,y) \in \R^2: \ x\geq 0, \ y\geq 0, \ x+y\leq 1, \ x\geq \tau (1-(x+y)), \ y\geq \tau (1-(x+y)) \}.
\]
The intersection of the lines $x=\tau(1-(x+y))$ and $y=\tau(1-(x+y))$ is the point $P=\left( \frac{1}{2+\frac{1}{\tau}}, \frac{1}{2+\frac{1}{\tau}} \right)$. See Figure \ref{fig:R}.
\begin{figure}
\begin{center}
  \includegraphics[width=0.6\textwidth]{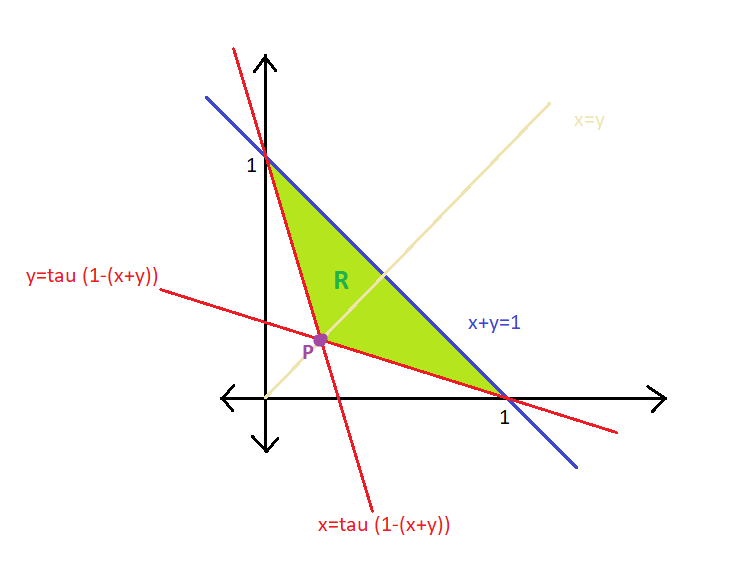}
\end{center}
\caption{The region $R$} \label{fig:R}
\end{figure}

\textbf{Claim:} \[1=\min \{x^{\beta}+y^{\beta}:  \ (x,y)\in R\}.\]

This is a standard calculation, but we provide the details. Let $g(x,y):=x^{\beta}+y^{\beta}$. Observe that since $g$ is increasing in $x$ and $y$, the minimum is reached on the red sides of the boundary of $R$ in Figure \ref{fig:R}. By symmetry, it is enough to study just one red side: we want to get the minimum of $g$ under the condition $y=1-(1+\frac{1}{\tau})x$ with $0 \leq x \leq \frac{1}{2+\frac{1}{\tau}}$. This is, the minimum of $h(x)= x^{\beta}+ \left(1-(1+\frac{1}{\tau})x \right)^\beta$ with $0 \leq x \leq \frac{1}{2+\frac{1}{\tau}}$. A calculation shows that the only critical point of $h'$ is a maximum, so the minimum is attained at some endpoint.





We have $h(0)=1$ and $h \left( \frac{1}{2+\frac{1}{\tau}} \right)=2 \left( \frac{1}{2+\frac{1}{\tau}} \right)^\beta=1$ (where the last equality holds by definition of  $\beta:= \frac{\log(2)}{\log(2+\frac{1}{\tau(C)})}$). This gives the claim.

Applying the claim to $x:=\frac{|U(L)|}{|A|}$ and $y:=\frac{|U(R)|}{|A|}$, we get \[|U(L)|^\beta + |U(R)|^\beta \geq |A|^\beta.\]

This means that changing $U(L)$ and $U(R)$ by $A$ in the covering $\mathcal{U}$ gives a new covering $\mathcal{U'}$ with one less element, by disjoint open sets, with smaller sum $\sum_i |U_i '|^\beta$.
Repeating this process we get that $\sum_i |U_i|^{\beta}\geq \sum_i |U_i '|^\beta \geq \cdots \geq |\conv (C)|^{\beta}$.
\end{proof}


\begin{observation}
Intuitively, thickness looks at the smallest part of the set, while Hausdorff dimension looks at the largest part (for example, if $C=\bigcup_{k \in \N} C_k$, then $\dim_H(C)=\sup_{k\in \N} \dim_H(C_k)$). So it is reasonable to apply Theorem \ref{dimthickness} to compact subsets of $C$. More precisely, we can define
the \emph{upper thickness} of a set $C$ as $\tilde{\tau}(C):=\sup_{A \text{ compact }\subseteq C} \tau(A)$. Then of course Theorem \ref{dimthickness} gives the same bound replacing $\tau(C)$ by $\tilde{\tau}(C)$.

As a simple instance of this, if a set has isolated points then the thickness is $0$, while the upper thickness ``gets rid'' of them and can take any value.

The upper thickness in usually larger than the thickness, but in some cases they may be equal.
\end{observation}

\section{Thickness and patterns in fractals}

In this section we investigate the connection between thickness and patterns in sparse sets.
\begin{definition}

We say that a set $C \subseteq \R^d$ contains a homothetic copy of $P$ if there exist $a \in \R\setminus\{0\}$ and $b \in \R^d$ so that $aP+b \subseteq C$.
\end{definition}

For example, an arithmetic progression of length $N$ in the real line is a homothetic copy of $\{1, \cdots, N\}$.

The following result is well known:
\begin{lemma}Any set $C \subseteq \R^d$ of positive Lebesgue measure contains homothetic copies of every finite set.
\end{lemma}
\begin{proof}

Let $P=\{a_1, \cdots, a_N\}$ be a finite set, and $R:=\max_{1\leq i \leq N}\| a_i\|$.

We know by the Lebesgue Density Theorem that almost every point $x \in C$ satisfies \[\lim_{r \to 0}\frac{\mathcal{L}^d(C \cap Q_r(x))}{\mathcal{L}^d(Q_r)}=1\] where $Q_r(x)$ is the cube with center $x$ and radius $r$. Fix an $x$ satisfying this. Then, there exists $r_0$ so that $\frac{\mathcal{L}^d(C \cap Q_{r_0}(x))}{\mathcal{L}^d(Q_{r_0})}>1-\frac{1}{10N}$.

Rescaling and translating the set $C$ and the cube $Q_{r_0}(x)$, we can assume $Q_{r_0}(x)=[0,1]^d$. Then, we know that $\mathcal{L}^d(C \cap [0,1]^d)>1-\frac{1}{10N}$.

It is enough to prove that
\[
\mathcal{L}^d \left( \bigcap_{1\leq i\leq N}(C-\frac{a_i}{10RdN}) \right)>0.
\]
Then, in particular, $\bigcap_{1\leq i\leq N}(C-\frac{a_i}{10RdN})$ is nonempty, and any point $y$ in the intersection satisfies that
\[
y+ (10RdN)^{-1} P \subset  C.
\]

Note that if $B \subseteq [0,1]^d$ and $\|v\|_2 \leq r$, then $\mathcal{L}^d\left((B-v) \cap [0,1]^d\right)\geq \mathcal{L}^d(B)-dr$.
Applying this to $B:=C\cap [0,1]^d$ and $v_i:=\frac{a_i}{10RdN}$ whose norm is smaller than $r:=\frac{1}{10dN}$, we get
\begin{align*}\mathcal{L}^d\left( (C \cap [0,1]^d - \frac{a_i}{10RdN}) \cap [0,1]^d \right) &\geq \mathcal{L}^d(C \cap [0,1]^d)-\frac{1}{10N}\\&\geq 1-\frac{2}{10N}.\end{align*}

If $A_1$, $\cdots$, $A_N \subseteq [0,1]^d$ satisfy $\mathcal{L}^d(A_i)\geq 1-\eps_i$ for all $i$, then
\[
\mathcal{L}^d \left( \bigcap_{1\leq i \leq N}A_i \right)\geq 1-\sum_{1\leq i \leq N} \eps_i.
\]

We apply this to $A_i:=(C \cap [0,1]^d - \frac{a_i}{10RdN}) \cap [0,1]^d$ and $\eps_i=\frac{2}{10N}$, to  get
\begin{align*}\mathcal{L}^d\left(\bigcap_{1\leq i\leq N}(C-\frac{a_i}{10RdN})\right)&\geq \mathcal{L}^d \left(\bigcap_{1\leq i\leq N} A_i \right)\\
&\geq 1-\sum_{1\leq i \leq N} \frac{2}{10N}= 1-\frac{1}{5}>0.\end{align*}
\end{proof}

Since positive Lebesgue measure guarantees homothetic copies of every finite set, it is natural to ask whether a weaker notion of size guarantees copies too. A natural notion of size to consider is Hausdorff dimension. However, Keleti \cite{Kel99} proved that there exists a compact set $C \subseteq \R$, with full Hausdorff dimension $1$, that does not contain any arithmetic progression of length $3$.  Afterwards, Keleti \cite{Kel08} improved this by constructing full Hausdorff dimensional compact sets in the real line avoiding homothetic copies of triplets in any given countable collection. Maga \cite{Maga} generalized this result to the complex plane. Máthé \cite{Mathe} constructed large Hausdorff dimensional compact sets avoiding polynomial patterns, in particular he generalized Keleti's result to countably many linear patterns. Finally, Yavicoli \cite{YavLinear} studied what happens ``in between'' positive Lebesgue measure and Hausdorff dimension $1$, by considering a more general notion of Hausdorff measures.

These facts indicate that Hausdorff measures and Hausdorff dimension cannot, by themselves, detect the presence or absence of patterns in sets of Lebesgue measure zero, even in the most basic case of arithmetic progressions. So, it is natural to seek a different notion of size that is able to detect patterns in sets of zero Lebesgue measure.

One of the ideas behind Keleti's construction is avoiding the given proportion everywhere at some scales of the construction. See Figure \ref{fig:Keleti}.
This picture happens on a ``zero density'' set of scales. So, the Hausdorff dimension can still be large (at ``almost all'' scales the set is large).
The notion of thickness is useful to avoid such examples: even one scale that looks like Figure \ref{fig:Keleti} this makes the thickness small.

\begin{figure}
\begin{center}  \includegraphics[height=3cm]{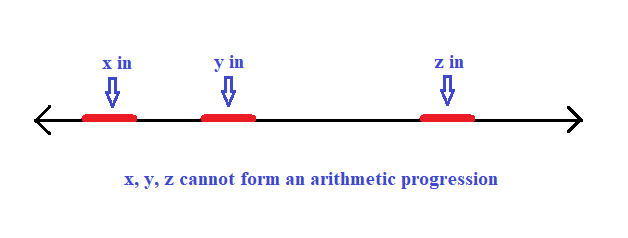}\end{center}
\caption{One step in Keleti's construction of a compact set that avoids progressions.} \label{fig:Keleti}
\end{figure}

Before studying patterns in relationship with thick sets, let me mention that Hausdorff dimension can be useful to detect some non-linear patterns (see \cite{orponen}), or to detect arithmetic progressions of length $3$ assuming additional Fourier decay hypotheses, which are often not explicit or hard to check (see \cite{CLP,HLP}). This suggests that it is natural to try to find explicit checkable conditions on a compact set that ensures that it contains arithmetic progressions, as well as other patterns.

Before studying arithmetic progressions, let us consider a different pattern -distances- using Newhouse's thickness and the Gap Lemma. We define the set of distances of a set $C \subseteq \R$ as \[\Delta (C):=\{|y-x|: \ x,y \in C \}.\]

\begin{lemma}
Let $C \subseteq \R$ be a compact set with $\conv(C)=[0,1]$ and $\tau(C)\geq 1$. Then, $\Delta (C)=[0,1]$.
\end{lemma}

\begin{proof}
We know that $\Delta (C) \subseteq [0,1]$, $0 \in \Delta (C)$ because $C\neq \emptyset$, and $1 \in \Delta (C)$ because $0,1\in C$.
It remains to see that any $t \in (0,1)$ belongs to $\Delta (C)$.

The sets $C$ and $C-t$ satisfy the hypotheses of the Gap Lemma: since $\conv (C)=[0,1]$ and $\conv (C-t)=[-t, 1-t]$, the convex hulls of $C$ and $C-t$ intersect and each set cannot be contained in a gap of the other set. Finally, using $\tau(C)\geq 1$ and the invariance of thickness under translations, we get $\tau(C)\tau(C-t)=\tau(C)^2\geq 1$.
By the Gap Lemma, there is $x \in C \cap (C-t)$, so $t= (x+t)-x \in C-C$ (which means that $t \in \Delta (C)$ because $t>0$).
\end{proof}

Now, we are going to see that, unlike Hausdorff dimension, set of large thickness contain $3$-term progressions:
\begin{proposition}
Let $C \subseteq \R$ be a compact set with $\tau (C)\geq 1$. Then, $C$ contains an arithmetic progression of length $3$.
\end{proposition}

\begin{proof}
Since the thickness of the compact set and arithmetic progression are invariant under homothetic functions, we can assume that $\conv(C)=[0,1]$.

This is the idea: if we prove that $C \cap \frac{C+C}{2}\neq \emptyset$, then there are $a,b,c \in C$ so that $c=\frac{a+b}{2}$.
Then, $\{a, \frac{a+b}{2}, b\} \subseteq C$. The problem is that a priori we could have $a=b$.  To avoid this issue we are going to consider two disjoint subsets of $C$ ($A$ and $B$ below), and prove that $C \cap \frac{A+B}{2} \neq \emptyset$.

Let $G:=(a_1, a_2)$ be the longest gap of $C$. Since $\tau(C)>0$, $0<a_1<a_2<1$.
The set $[0,1]\setminus G$ consists of two intervals, $L=[0,a_1]$ and $R=[a_2,1]$. We can assume that \[a_1\leq 1-a_2,\] otherwise, we would work with $-C+2$ instead.

Let $A:=C \cap [0,a_1]$ and $B:=C \cap [a_2,1]$. Since we want to show $C \cap \frac{A+B}{2} \neq \emptyset$, we want to understand $A+B$.

\textbf{Claim: $A+B=[a_2, 1+a_1]$}

Clearly, $A+B \subseteq [0,a_1]+[a_2,1]=[a_2,1+a_1]$, so we need to see the other inclusion. We have that
\[|\conv(-A)|=|\conv(A)|=a_1\] and
\[\tau(-A)=\tau(A)=\tau(C\cap [0,a_1])\geq \tau(C)\geq 1.\]
Observe that $\tau(C\cap [0,a_1])\geq \tau(C)$ holds because $G_1$ is the largest gap (otherwise this may not be true, since in general thickness is not well behaved with respect to intersections).

Analogously with $B-t$,
\[|\conv(B-t)|=1-a_2 \text{ and } \tau(B-t)\geq 1.\]

We are going to apply the Gap Lemma to $-A$ and $B-t$ for any $t \in [a_2, 1+a_1]$. Let us see that the assumptions are satisfied.

Note that
\begin{align*}
&\conv(-A)\cap\conv(B-t)\neq \emptyset\\
&\Leftrightarrow \exists \text{ an endpoint of } -A \text{ between the endpoints of } B-t \\
&\Leftrightarrow -a_1 \in [a_2-t,1-t] \text{ or } 0\in [a_2-t,1-t]\\
&\Leftrightarrow t\in [a_1+a_2, 1+a_1] \text{ or } t\in[a_2,1]\\
&\Leftrightarrow t\in [a_2,1+a_1],
\end{align*}
where the last equivalence holds since we assume $a_1+a_2 \leq 1$.

Then:
\begin{itemize}
\item $-A$ is not contained in a gap of $B-t$.
This is true, because since $\tau(C)\geq 1$, $|\conv(-A)|\geq |G|\geq |\text{any gap of }B-t|$. Analogously $B-t$ is not contained in a gap of $A$.
\item  $\conv(-A)\cap\conv(B-t)\neq \emptyset$ since we are considering values of $t \in [a_2, 1+a_1]$.
\item $\tau(-A)\tau(B-t)\geq 1$.
\end{itemize}

Then, the Gap Lemma yields that for all $t\in [a_2,1+a_1]$ there is $x_t\in (-A)\cap (B-t)$, and then $t =(-x_t)+(x_t+t)\in A+B$, giving the claim.

\textbf{Claim: $C \cap \frac{A+B}{2} \neq \emptyset$}

We are going to prove  that in fact $a_2 \in \frac{A+B}{2}$ (note that $a_2 \in C$).

Since $A+B=[a_2, 1+a_1]$, we know that
\[a_2 \in \frac{A+B}{2} \Leftrightarrow a_2 \leq \frac{1+a_1}{2}.\]

Where are the pairs $(a_1,a_2)$ that we are working with?
\begin{itemize}
\item We have $a_1=|\conv(A)|\geq |G|=a_2-a_1$, so that $2a_1 \geq a_2$.
\item We are assuming that $a_1 \leq 1-a_2$.
\item $0<a_1<a_2<1$
\end{itemize}

\begin{figure}
\begin{center}
  \includegraphics[width=0.6\textwidth]{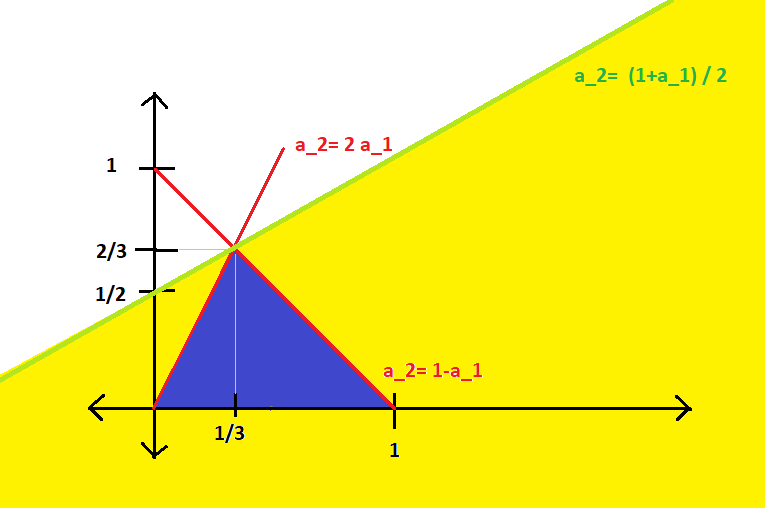}
\end{center}
\caption{The region of possible pairs $(a_1,a_2)$.} \label{fig:region123}
\end{figure}

As we can see in Figure \ref{fig:region123}, the blue region of pairs $(a_1,a_2)$ is contained in the yellow region $\{ a_2 \leq \frac{1+a_1}{2}\}$. In particular, we must have $a_2 \leq \frac{1+a_1}{2}$, so $a_2 \in \frac{A+B}{2}$, as claimed.

We have seen that $C$ contains an arithmetic progression $\{ a, (a+b)/2,b\}$. This is indeed a non-degenerate progression since $a\in A, b\in B$ and $A$ and $B$ are disjoint, so the proof is complete.
\end{proof}


What about longer arithmetic progressions? For example, what is the length of the longest arithmetic progression that is contained in the middle-$\eps$ Cantor set?

\begin{lemma} The middle-$\eps$ Cantor set $M_{\eps}$ does not contain an arithmetic progression of length $\lfloor \frac{1}{\eps}\rfloor +2$.\end{lemma}

\begin{proof}
Take $\lambda:=\frac{1-\eps}{2}$. The length of any interval of step $k$ of the construction is $\lambda^k$ (for $k\geq0$), and the length of any gap of the step $k$ of the construction is $\lambda^{k-1}\eps$ (for $k\geq 1$).
Assume that there is an arithmetic progression $a+\Delta j$, $j=1,\ldots, N$ (with $\Delta>0$) contained in $M_{\eps}$. Then, by self-similarity, there is a first step $k \in \N\cup\{0\}$ so that the arithmetic progression is contained in an interval of step $k$ but splits in the step $k+1$.
Hence, $|(a+\Delta N)-(a+\Delta)|=\Delta (N-1)\leq \lambda^{k}$ and $\Delta \geq \lambda^k \eps$.
Therefore, $\lambda^k \geq \lambda^k \eps (N-1)$, so $\frac{1}{\eps}+1\geq N$.
\end{proof}

Getting lower bounds for the length of the longest arithmetic progression in $M_{\eps}$ is more difficult. Broderick, Fishman and Simmons \cite{BFS} proved the following result:
\begin{theorem}[Broderick, Fishman, Simmons] \label{thm:BFS}
For $\eps>0$ sufficiently small, the $M_{\eps}$ middle-$\eps$ Cantor set contains an arithmethic progression of length $c \frac{\frac{1}{\eps}}{\log(\frac{1}{\eps})}$, where $c$ is a very small constant.
\end{theorem}

By the two previous results, we know that for $\eps>0$ sufficiently small,
\[\frac{\frac{1}{\eps}}{\log(\frac{1}{\eps})} \lesssim \text{ longest AP contained in }M_{\eps} \lesssim \frac{1}{\eps}.\]
The precise asymptotic behaviour remains an open problem.

We won't give a proof of Theorem \ref{thm:BFS}. The very rough idea behind the proof is:
\begin{align*}
&C \text{ contains an arithmetic progression of length } n \text{ and gap length }\lambda \\
&\Leftrightarrow \text{there is } x \text{ so that } x+k\lambda \in C \ \forall 1\leq k\leq n\\
&\Leftrightarrow \bigcap_{1 \leq k \leq n} (C-k \lambda)\neq \emptyset
\end{align*}
Then, the existence of arithmetic progressions of length $n$ is reduced to proving that intersections of certain $n$ sets is nonempty. Unfortunately, the Gap Lemma does not generalize in any simple way to intersections of $3$ or more sets, and for this reason the authors use a different approach: the Potential Game, which is a game of Schmidt type.

The classical Schmidt game was defined in 1966 by Wolfgang Schmidt to study badly approximable numbers, and since then many variants of the original game have been developed, mainly  to study problems in diophantine approximation.

As a general idea, the Potential Game is a game in which there are certain rules and two players: Bob who decides where we are going to \textbf{zoom-in}, and Alice who decides what to \textbf{erase} there. Bob has limits on how far to zoom in, and Alice has limits on how much to erase. And there are special sets called \textbf{winning sets} which are subsets of the ``board game''. A set $W$ is winning if Alice has a strategy guaranteeing that if she did not erase the limit point of convergence for Bob's moves during the game, then that point belongs to $W$.
Being a winning set (for certain parameters) can be considered as another notion of ``large size'' for the set.

Broderick, Fishman and Simmons showed that a slight modification of a middle-$\eps$ Cantor set is a winning set with certain parameters. Then they use that intersections of winning sets are winning (for certain other parameters), and prove and use a result that gives a (positive) lower bound for the Hausdorff dimension of a winning set inside certain balls. In particular, the intersection is nonempty.

As a remark, winning sets for the classical Schmidt's game and many variants have full Hausdorff dimension. This is not the case for the potential game (with fixed parameters). This makes it useful to study fractal sets that do not have full Hausdorff dimension.

We will define now the Potential Game in a restricted context (on the real line, where Alice is able to erase neighborhoods of points; the game can be extended to higher dimensions and more general sets).
\begin{definition}[Potential game in $\R$]
Given $\alpha, \beta, \rho >0$ and $c \geq 0$, Alice and Bob play the $(\alpha, \beta, c, \rho)$-Potential Game in $\R^d$ under the following rules:
\begin{itemize}
\item For each $m \in \N_{0}$ Bob plays first, and then Alice plays.
\item On the $m$-th turn, Bob plays a closed ball $B_m:=B[x_m , \rho_m ]$. The first ball must satisfy $\rho_0 \geq \rho$. The following moves must satisfy $\rho_{m}\geq \beta \rho_{m-1}$ and $B_m \subseteq B_{m-1}$ for every $m \in \N$.
\item On the $m$-th turn Alice responds by choosing and erasing a finite or countably infinite collection $\mathcal{A}_m=\{A_{\rho_{i,m}}\}_i$ of balls with radii $\rho_{i,m}>0$. Alice's collection must satisfy:
\begin{align*}
\sum_{i} \rho_{i,m}^c \leq (\alpha \rho_m )^c & \text{ if } c>0 \\
\rho_{1,m} \leq \alpha \rho_m & \text{ if } c=0 \text{ (in this case Alice can erase just one set)}.
\end{align*}
\item Alice is allowed not to erase any set, or equivalently to pass her turn.
\item Bob must ensure that $\lim_{m \to \infty} \rho_m =0$.
\end{itemize}

There exists a single point
\[
\{ x_{\infty} \} = \bigcap_{m \in \N_0} B_m
\]
called the \textbf{outcome of the game}.
\item We say a set $S \subset \R^d$ is an $(\alpha, \beta, c, \rho)$-\textbf{winning set} if Alice has a strategy guaranteeing that
\[
\text{ if } x_{\infty} \notin \bigcup_{m \in \N_0} \bigcup_i A_{\rho_{i,m}} \,\, \text{\, then } x_{\infty} \in S.
\]
\end{definition}

The Potential Game has several elementary but very useful properties.
\begin{lemma}[Countable intersection property]\label{Countable intersection property}
Let J be a countable index set, and for each $j \in J$ let $S_j$ be an $(\alpha_j , \beta, c, \rho)$-winning set, where $c>0$.

Then, the set $S:= \bigcap_{j \in J} S_j$ is $(\alpha ,\beta, c, \rho)$-winning where $\alpha^c = \sum_{j \in J} \alpha_j^c$ (assuming that the series converges).
\end{lemma}
To see this, it is enough to consider the following strategy for Alice: in the turn $m$ she plays the union over $j$ of all the strategies of turn $m$.
For each $j$ we know that for each turn $m$ we have $\sum_i \rho_{i,m}(j)^c \leq (\alpha_j \rho_m)^c$. Now, we can see that playing all the strategies together is legal: in turn $m$ we have \[\sum_j \sum_i \rho_{i,m}(j)^c \leq \sum_j (\alpha_j \rho_m)^c= (\sum_j \alpha_j^c) \rho_m^c=\alpha^c \rho_m^c.\]

\begin{lemma}[Monotonicity]\label{Monotonicity}
If $S$ is $(\alpha , \beta, c, \rho)$-winning and $\tilde{\alpha} \geq \alpha$, $\tilde{\beta} \geq \beta$, $\tilde{c} \geq c$ and $\tilde{\rho} \geq \rho$, then $S$ is $(\tilde{\alpha} , \tilde{\beta}, \tilde{c}, \tilde{\rho})$-winning.
\end{lemma}

Indeed, one can check that Alice can answer in the $(\tilde{\alpha} , \tilde{\beta}, \tilde{c}, \tilde{\rho})$-game using her strategy  from the $(\alpha , \beta, c, \rho)$-game.


\begin{lemma}[Invariance under similarities]\label{Invariance under similarities}
Let $f:\R^d \to \R^d$ be a  similarity with contraction ratio $\lambda$. Then a set $S$ is $(\alpha , \beta, c, \rho)$-winning if and only if the set $f(S)$ is $(\alpha , \beta, c, \lambda \rho)$-winning.
\end{lemma}
This follows by mapping Alice's strategy by $f$.

In \cite{Y21}, I established a new connection between Schmidt's Games and thickness in the real line and generalized the result by Broderick, Fishman and Simmons:
\begin{theorem} \label{thm:y21}
Let $C\subset\mathbb{R}$ be a compact set. Then $C$ contains a homothetic copy of every set $P$ with at most
\[
N(\tau):=\left\lfloor\frac{\log (4)}{4 e (720)^2}\frac{\tau}{\log (\tau)}\right\rfloor
\]
elements.
Moreover, for each such set $P$, the compact set $C$ contains $\lambda P+x$ for some $\lambda>0$ and a set of $x$ of positive Hausdorff dimension.
\end{theorem}
Note that this result gives non-trivial information only when $N(\tau)\ge 3$, which requires the thickness to be larger than some large absolute constant. The main usefulness of the theorem is for large values of $N(\tau)$. Theorem \ref{thm:y21} generalizes the result of Broderick, Fishman and Simmons, since $M_{\eps}$ is a compact set with thickness $\sim 1/\eps$ and an arithmetic progression is a homothetic copy of $\{1,\ldots,N\}$.

Let us see the main ideas behind the proof of Theorem \ref{thm:y21}. Given a finite set $P:=\{p_1, \cdots, p_n\}$,
\begin{align*}
&C \text{ contains a homothetic copy of  } P \\
&\Leftrightarrow \exists \lambda\neq0 , \text{ there is } x \text{ so that } x+\lambda p_k \in C \ \forall 1\leq k\leq n\\
&\Leftrightarrow \exists \lambda\neq0 \text{ so that } \bigcap_{1\leq k \leq n} (C-\lambda p_k)\neq \emptyset.
\end{align*}
Then, to guarantee a pattern of size $n$ we need to check that certain intersection of $n$ sets is nonempty (in fact, the proof shows that the intersection has positive Hausdorff dimension). We know from Lemma \ref{Countable intersection property} that winning sets have certain stability under intersections. It is not obvious that winning sets intersect a given interval, but Broderick, Fishman and Simmons \cite[Theorem 5.5]{BFS} proved that, depending on the parameters, the intersection of a winning set with an interval not only is nonempty but has positive Hausdorff dimension. While \cite[Theorem 5.5]{BFS} involves some non-explicit constants, in the context relevant to Theorem \ref{thm:y21} this result was made completely explicit in \cite[Theorem 19]{Y21}.

What is remaining to prove Theorem \ref{thm:y21} is the link between thick sets and winning sets. This is provided by the following result:
\begin{proposition}\label{TauWinning}
Let $C$ be a compact set with $conv (C)=[0,1]$ and $\tau:= \tau(C)>0$. Then $S:=(-\infty,0) \cup C \cup (1, +\infty)$ is $\left( \frac{1}{\tau \beta}, \beta, 0, \frac{\beta}{2}\right)$-winning for all $\beta \in (0,1)$.
\end{proposition}

\begin{proof}
In order to prove that a set $S$ is winning, we have to see that Alice is able to erase the complement of $S$ where Bob is zooming-in.

If Bob plays $B$, how does Alice respond? Let $(G_n)_n$ be the sequence of complementary open gaps of $S$, ordered by non-increasing length.

\textbf{Alice's strategy}: If there exists $n \in \N$ such that $B$ intersects $G_n$ and $|B| \leq \min \{ |L_n| ,|R_n| \}$, then Alice erases $G_n$ if it is a legal movement. In any other case (if $B$ does not intersect any gap of $S$ or if $|B|>\min\{ |L_n|, |R_n|\}$), Alice does not erase anything.

To show that this strategy is winning, suppose that Alice does not erase $x_{\infty}$ during the game. We want to see that $x_{\infty} \in S$.
Let us make a counter-assumption that $x_{\infty} \notin S$. Then there exists $n$ such that $x_{\infty} \in G_n$.
We will show that Alice erases $G_n$ at some stage of the game (which is a contradiction).
By definition, $x_{\infty} \in B_m$ for all $m \in \N_0$, and we assumed $x_{\infty} \in G_n$, so
\[x_{\infty} \in B_m \cap G_n \text{ for all } m \in \N_0.\]

Since $\tau>0$, we have that $\min \{ |L_n|,|R_n| \}>0$, and we also know that $\lim_{m \to \infty}|B_m|=0$. So, taking $m_n$ to be the smallest integer such that
\[\min \{ |L_n|,|R_n| \} \geq |B_{m_n}|,\]
we know that
\begin{equation*}\label{eq}B_{m_n}
\cap G_n \neq \emptyset \text{ and then }B_{m_n} \cap G_k = \emptyset \ \forall 1 \leq k <n.
\end{equation*}
Then $|B_{m_n}|\geq \beta \min \{ |L_n|,|R_n| \}$. Indeed:
\[\text{If }m_n=0, \text{ then } |B_0|=2\rho_0 \geq 2 \rho=\beta \geq \beta \min \{ |L_n|,|R_n| \}.\]
\[
\text{If } m_n>0, \text{ then } |B_{m_n}|\geq \beta |B_{m_n-1}|> \beta \min \{ |L_n|,|R_n| \}.
\]

Recall that we proved that $B_{m_n} \cap G_n \neq \emptyset \text{ and }B_{m_n} \cap G_k = \emptyset \, \forall 1 \leq k <n$, and $\beta \min \{ |L_n|,|R_n| \} \le |B_{m_n}| \le \min \{ |L_n|,|R_n| \}$. Hence,
\[
|G_n|\leq \frac{1}{\tau} \min \{ |L_n|,|R_n| \} \leq \frac{1}{\tau \beta} |B_{m_n}|=\alpha |B_{m_n}|.
\]
Since $G_n$ is the \textbf{first} gap intersecting $B_{m_n}$, the gap $G_n$ is uniquely defined (there aren't two gaps that Alice should erase in the same turn). In conclusion,  it is legal for Alice to erase $G_n$ in the $m_n$-th turn, and her strategy specifies that she does so.
\end{proof}

Sketch of proof of Theorem \ref{thm:y21}:
We can assume without loss of generality that $\conv (C)=[0,1]$, and also that the pattern with $n$ elements is $\{b_1, \cdots ,b_n\} \subseteq [0,\frac{1}{8}]$.
We define \[S_i:=(-\infty,-b_i) \cup (C-b_i) \cup (1-b_i, +\infty).\]
Using Propositions \ref{TauWinning}, \ref{Monotonicity}, \ref{Invariance under similarities} and \ref{Countable intersection property}  we get that $S:=\bigcap_{1\leq i \leq n} S_i$ is $\left( \frac{n^{\frac{1}{c}}}{\tau \beta}, \beta, c, \frac{\beta}{2}\right)$-winning for all $\beta \in (0,1)$ and all $c>0$.
We define $\alpha:=\frac{1}{\tau \beta}$, and take $\beta:=\frac{1}{4}$, $c:=1-\frac{1}{\log(\alpha^{-1})}$, and $B:=[\frac{3}{8}, \frac{5}{8}]$ which is an interval of length $\frac{1}{8}=\frac{\beta}{2}=:\rho$.
Then, applying \cite[Theorem 19]{Y21} (which is a very technical result from where we get the constant $\frac{\log (4)}{4 e (720)^2}$), one gets the following condition:
\[\text{dim}_{\rm H} (S \cap B)>0\]
if \begin{equation*} n \alpha^c \leq \frac{1}{720^2} (1 − \beta^{1−c}). \end{equation*}

So, to guarantee the presence of a homothetic copy of a set of size $n$, it is sufficient that $n$ satisfies the hypothesis of the Theorem.

For those values of $n$, we have seen $\text{dim}_{\rm H}(S \cap B)>0$. For each $x\in S\cap B$, using $0 \leq b_i \leq \frac{1}{8}$, we have $$x+b_i \in (B+b_i)\cap(S+b_i) \subset \left[ \frac{3}{8}, \frac{6}{8} \right]\cap \left( (-\infty, 0) \cup C \cup (1, +\infty) \right).$$
Since $[\frac{3}{8}, \frac{6}{8}]$ is disjoint from $(-\infty, 0)$ and $(1, +\infty)$, we have that $x+b_i \in C$.

So $x+ \{ b_1, \cdots , b_n \}$ is a translated copy of the given finite set, which is contained in $C$.

\section{Extensions of thickness to higher dimensions}

The definition of thickness and the proof of the Gap Lemma strongly use the order structure of the reals. It has been an open problem to find a satisfactory extension to higher dimensions. Some of the existing attempts include:

\begin{itemize}
\item S. Biebler \cite{Biebler} defined a notion of thickness that applies to dynamically defined sets in the complex plane.
\item De-Jun Feng and Y. Wu  \cite{DJF} defined another notion of thickness that is useful to study arithmetic sums (they did not however study the Gap Lemma).
\end{itemize}

In the rest of the article, I describe two more recent notions developed in \cite{FalconerYavicoli, y22}.

\subsection{Thickness in $\R^d$ (useful for cut-out type sets)}

With Kenneth Falconer \cite{FalconerYavicoli}, we gave a different definition of thickness in $\R^d$ that is useful for sets of cut-out type (see Figure \ref{Fig:kenneth} for an example).

Let $E \subset \R^d$ be an open path-connected set whose complement is a non-empty compact set, and let $(G_n)_n$ be a sequence of disjoint path-connected open sets contained in $E^C$. In the special case $d=1$ the set $E$ is formed by the union of two disjoint open unbounded sets (this is the only case in which $E$ is not a path-connected set).
We say that $E$ is the external component and $(G_n)_n$ is the sequence of gaps associated to the cut-out type set \[C:=E^C \setminus \bigcup_n G_n.\]
We can consider the sequence of gaps ordered by non-increasing diameter.

\begin{figure}
\includegraphics[width=8cm,height=4cm]{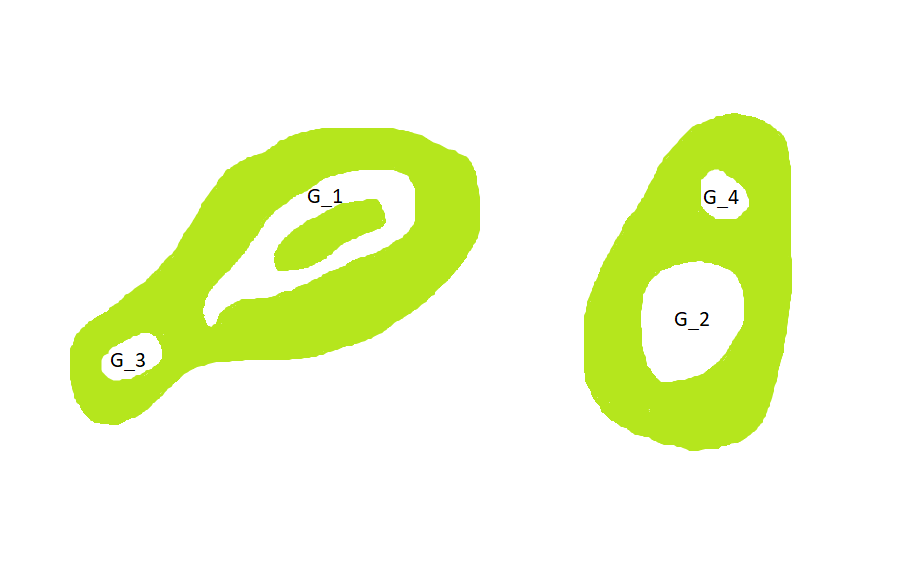}
\caption{Step 4 of construction of a cut-out type set}\label{Fig:kenneth}
\end{figure}

We define
\[\tau(C):=\inf_{n \in \N} \frac{\dist (G_n, \bigcup_{1\leq i \leq n-1}G_i \cup E)}{\diam (G_n)}.\]
provided that there is at least one gap $G_n$.
In case there are not gaps, we define
\begin{equation}
\tau (C):= \left\{ \begin{array}{lcc}
             +\infty &   \text{if}  & C^{\circ} \neq \emptyset \\
             \\ 0 &  \text{if}  & C^{\circ} = \emptyset
             \end{array} \right .
 \end{equation}

This definition has certain advantages: it can take any value in $[0, \infty]$, it is invariant under homothetic functions, and on the real line it coincides with the classical one. In \cite{FalconerYavicoli} we obtained a first extension to the Gap Lemma to $\R^d$, and also studied intersections of countably many thick sets.
But there is a significant drawback: sets of positive thickness look like cut-out sets (``poking holes''). In particular, totally disconnected sets, which are of special interest in dynamical systems and fractal geometry, have zero thickness and so this notion is not suitable to study them.

\subsection{Thickness in $\R^d$ (useful in general, even for totally disconnected sets)}

In \cite{y22}, I was able to give a definition of thickness that is useful also for many totally disconnected sets, and proved a higher dimensional Gap Lemma, among other results.

From now on, we will work on $\R^d$ equipped with the distance $\dist_\infty$ coming from the infinity norm, and all cubes (balls for this distance) will be closed. Recall that $\dist_\infty (x,y):=\|x-y\|_{\infty}:=\max_{1 \leq i \leq d} | x_i - y_i|$. In fact the notion of thickness and the results extend to any norm, but the infinity norm is the most useful one because cubes can be used to efficiently pack larger cubes (as opposed to, for example, Euclidean balls).

Given a word $I$ (a finite sequence of natural numbers) we denote the length of $I$ by $\ell(I) \in \N_0$. We say that $(S_I)$ is a \emph{system of cubes} for a compact set $C$ if
\[
C=\bigcap_{n \in \N_0} \bigcup_{\ell(I)=n}S_I
\hspace{1cm}\text{(Moran construction)},\]
where
\begin{itemize}
\item Each $S_I$ is a cube and contains $\{ S_{I,j} \}_{1 \leq j \leq k_I}$. No assumptions are made on the separation of the $S_{I,j}$.
\item For every infinite word $i_1, i_2, \cdots$ of indices of the construction, \[\lim_{n \to +\infty} \rad (S_{i_1,i_2, \cdots, i_n})=0.\]
\end{itemize}
Note that any compact set arises from multiple systems of cubes but, for any $d\ge 2$, there is no canonical way to choose a system of cubes for a given set.

Let us fix a system of cubes $(S_I)$ for $C$. The main issue with extending the notion of thickness to compact sets in $\R^d$ is that there isn't a suitable notion of ``gap''. We use the following notion as a substitute of gap-size (see Figure \ref{hI}):
\[h_I:= \max_{x \in S_I} \dist_{\infty} (x, C).\]
\begin{figure}\includegraphics[width=7cm]{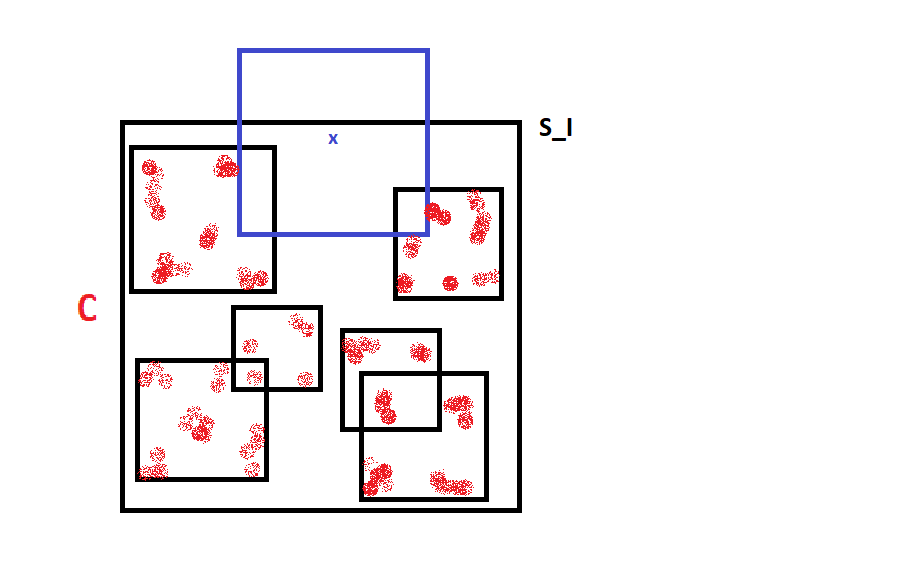}\caption{The radius $h_I$ of the blue square is the substitute of the notion of gap-size}\label{hI}\end{figure}

Hence $h_I$ is characterized by the properties that if $z \in S_I$ and $r \in [h_I, +\infty)$, then the closed cube $B(z, r)$ intersects $C$ but,
on the other hand, for every $r <h_I$ there is  $z \in S_I$ such that $B(z,r)\cap C=\emptyset$.


\begin{definition}[Thickness of $C$ associated to the system of cubes $\{S_I\}_I$] \label{def:thickness-Rd}
\[\tau(C, \{S_I\}_I):= \inf_{n \in \N_0} \inf_{\ell(I)=n} \frac{\min_i \rad(S_{I,i})}{h_I}.\]
\end{definition}

This definition preserves some of the basic properties of Newhouse thickness. Indeed, on the real line it agrees with Newhouse's thickness (for the natural system of cubes arising in Newhouse's definition). Thickness is invariant under homothetic functions (for the system of cubes obtained via mapping by the corresponding function). As in the real line, if a set has large thickness then it also has large Hausdorff dimension, assuming that each cube has at least $M_0\ge 2$ non-overlapping children (a mild and reasonable assumption that is automatic on $\R$):
\[
\dim_H (C) \geq \frac{d}{1+\frac{\log (1+\frac{1}{\tau})}{\log (M_0)}}.
\]
See \cite[Lemma 4]{y22}.

Unlike Newhouse's definition, the above notion of thickness depends on the system of cubes used to generate $C$. If the cubes provide a ``bad approximation'', then the resulting value for $\tau$ can be artificial. To understand this, consider the example $C:=\{0\} \subseteq \R$ and the system of cubes that has just one cube in each level: $\{[-\frac{1}{n},\frac{1}{n}]\}_{n \in \N}$. Then $h_n=\frac{1}{n}$, so
\[\tau \left( C, \{[-1/n,1/n\}_{n \in \N} \right)=\inf_{n \in \N}\frac{\frac{1}{n+1}}{\frac{1}{n}}=\inf_{n \in \N}\frac{n}{n+1}=\frac{1}{2}.\]
But intuitively one expects the thickness of a singleton to be $0$.

In order to state the Gap Lemma in $\R^d$, we need an additional condition that says that the children of any cube in the system are ``well spread out''. Part of the motivation for this definition is to avoid pathological systems of cubes such as the above example.
\begin{definition}
We say that a system of cubes $\{S_I\}_I$ is \emph{$r$-uniformly dense} if for every $I$, for every cube $B \subseteq S_I$ with $\rad (B)\geq r \ \rad (S_I)$, there is a child $S_{I,i} \subseteq B$.
\end{definition}

Let us consider an example. Fix $n \ge 2$ and $\ell \in (0, \frac{2}{n})$. We consider a \textbf{corner Cantor set} $C = C_{\ell,n} \subseteq (\R^d, \dist_{\infty})$ as in Figure \ref{fig:corner}. Let $g=\frac{2-n\ell}{n-1}$. One can check that the thickness is given by $\tau(C) =\frac{\ell}{g}=\frac{\ell (n-1)}{2-n\ell}$, and the set is $r:=\frac{1}{2}( 2 \ell + g)= \left(\ell +\frac{2-n\ell}{2(n-1)}\right)$-uniformly dense.

\begin{figure}
\begin{center} \includegraphics[width=5cm]{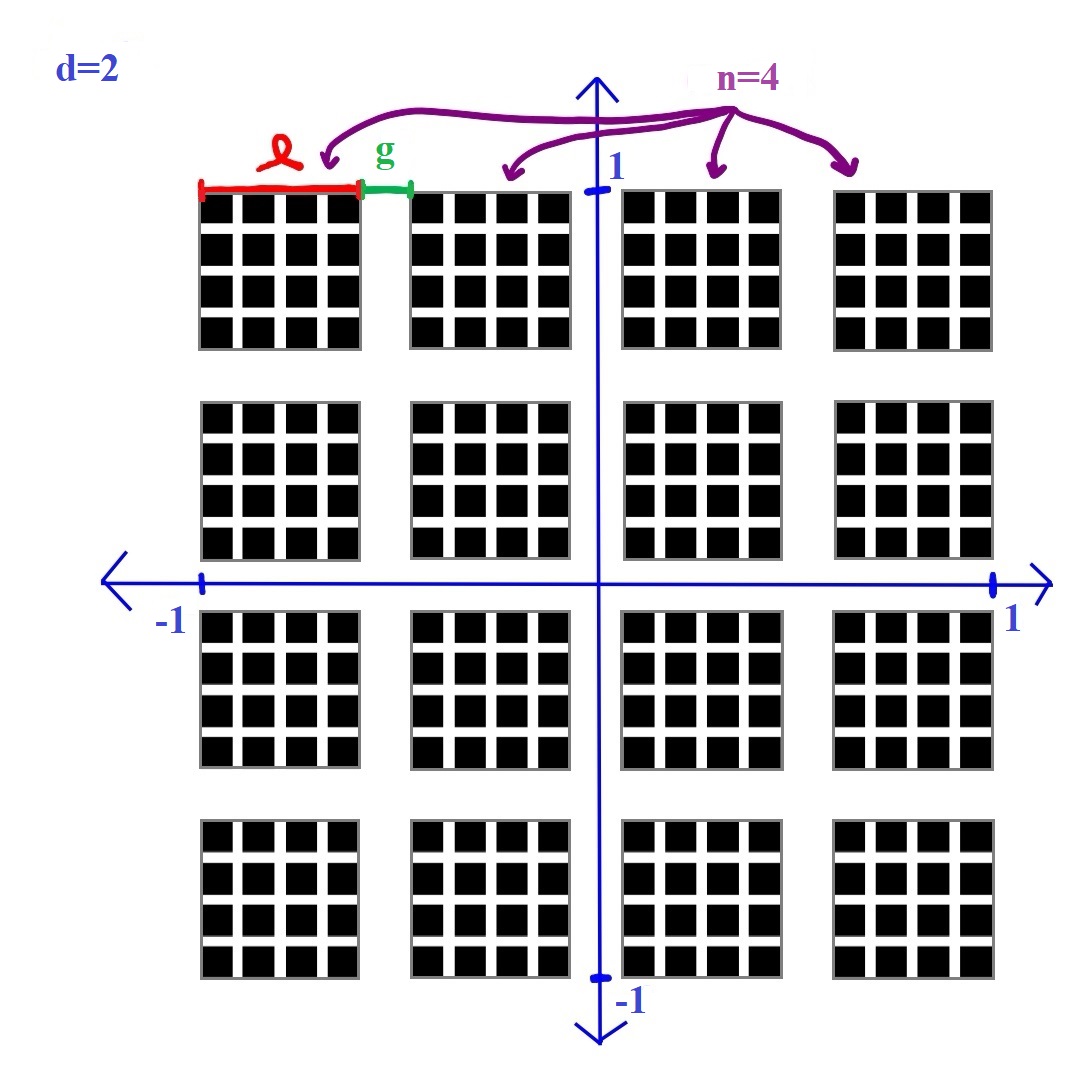}\end{center}
\caption{A corner Cantor set}\label{fig:corner}
\end{figure}

\begin{theorem}[Higher dimensional Gap Lemma \cite{y22}]\label{Theo:GapLemma}
Let $C^1$ and $C^2$ be two compact sets in $(\R^d, \dist_{\infty})$, generated by systems of cubes $\{S_I^1\}_I$ and $\{S_L^2\}_L$ respectively, and fix $r \in (0, \frac{1}{2})$. Assume:
\begin{enumerate}[(i)]
\item \label{it:i} $\tau(C^1, \{S_I^1\}_I) \tau(C^2, \{S_L^2\}_L)\ge \frac{1}{(1-2r)^2}$,
\item \label{it:ii} $\{S_I^1\}_I$ and $\{S_L^2\}_L$ are $r$-uniformly dense.
\item \label{it:iii} $C^1 \cap (1-2r) S_{\emptyset}^2 \neq \emptyset$ and $\rad (S_{\emptyset}^1) \geq r \ \rad (S_{\emptyset}^2)$
\end{enumerate}
Then, \[C^1 \cap C^2 \neq \emptyset.\]
\end{theorem}

Some remarks on this statement are in order. Unlike Newhouse's Gap Lemma, we need the additional ``uniform denseness'' assumption. As mentioned above, this is partly to above systems of cubes that yield an artificially large value for the thickness.

Assumption \eqref{it:iii} is quite mild; it can be seen as a stronger version of the hypotheses (in the original Gap Lemma) that the convex hulls intersect and that each Cantor set is not contained in a gap of the other.

There is a balance between assumptions \eqref{it:i} and \eqref{it:ii}: the first condition is stronger when $r$ is close to $\frac{1}{2}$, and the second condition is stronger when $r$ is close to $0$. Note also that as $r\to 0$, assumption \eqref{it:i} reduces back to the product of the thicknesses being larger than $1$, as in the original Gap Lemma.

An important feature of Newhouses's Gap Lemma is that the hypotheses are robust under perturbations of the Cantor sets. This is also the case for Theorem \ref {Theo:GapLemma}. For example, the assumptions are robust under $C^1$ perturbations whose derivatives are close to the identity applied to the sets, and if the sets are self-homothetic, they are also robust under perturbations of the generating iterated function system - see \cite[Lemmas 7 and 8]{y22} for details.

Recall that S. Biebler \cite{Biebler} defined a notion of thickness and proved a Gap Lemma for a class of dynamically defined compact sets in the plane. Even in this restricted context, Theorem \ref{Theo:GapLemma} applies in many more cases (roughly speaking, Biebler has more restrictive version of each of the assumptions). The definition of thickness and proof of the Gap Lemma are also significantly simpler than Biebler's.

\subsection{An application to directional distance sets}

Given $E \subseteq \R^d$, the distance set of $E$ is
\[
\Delta (E):=\{\|x-y\|_2 : \ x,y \in E\} \subseteq [0,\infty).
\]
It is a major open problem to understand the relationship between the sizes of $E$ and $\Delta(E)$.
\begin{conjecture}[Falconer's distance conjecture]
If $E$ is a compact set with $\dim_H(E)>d/2$, then $\Delta(E)$ has positive Lebesgue measure.
\end{conjecture}

There are many partial results and variants. One of them involves  investigating conditions under which the distance set has non-empty interior. The best result in this direction is due to Mattila and Sjölin:

\begin{theorem}[Mattila and Sjölin \cite{MattilaSjolin}]
If $E$ is a compact set with $\dim_H(E)>\frac{d+1}{2}$, then the distance set $\Delta(E)$ has non-empty interior.
\end{theorem}
This result does not give quantitative bounds on the size of a ball in $\Delta(E)$, nor does it guarantee that $0$ is an interior point of $\Delta(E)$.

We know that sets with large thickness have large Hausdorff dimension. So one can ask whether we can get stronger consequences for sets of large thickness. We will see that the Gap Lemma provides such a result.

For a fixed  direction $v \in S^{(d-1)}$, we say that $t\geq 0$ is a distance between points of $E$ in direction $v$ if there are $e_1$ and $e_2$ in $E$ such that $e_1-e_2=tv$. We define $\Delta_v (E)$ as the set of all distances between points in the set $E$ in direction $v$. Of course, $\Delta_v (E) \subseteq \Delta(E)$.

Applying the Gap Lemma to $E$ and $E-tv$, one can get that $\Delta_v(E)$ contains an explicit uniform interval for any direction $v$:
\begin{corollary}
Let $E=\bigcap_{n \geq 0} \bigcup_{\ell (I)=n}S_I$ be a compact set in $\R^d$ so that there exists $r \in (0, \frac{1}{3}]$ satisfying
\begin{itemize}
\item $\tau (C, \{S_I\}_I) \geq \frac{1}{1-2r}$
\item $\{S_I\}_I$ is $r$-uniformly dense with respect to $C$
\end{itemize}
Then, there is $a>0$ (depending only on $r$ and the radius of $S_\varnothing$) such that for any direction $v \in S^{(d-1)}$ we have
\[
[0,a]\subseteq \Delta_v (C).
\]
\end{corollary}

\begin{proof}
We can assume without loss of generality that $S_{\emptyset}=B[0,1]$.

Let $v$ be any vector in $S^{(d-1)}$. We are going to show that the sets $C$ and $C+tv$ satisfy the hypothesis of the Gap Lemma (Theorem \ref{Theo:GapLemma}) for $t \in [\frac{-2r}{1-2r}, \frac{2r}{1-2r}]$. Then we will have $C \cap (C+tv) \neq \emptyset$ for any  $t \in [0, \frac{2r}{1-2r}]$, and so $[0,\frac{2r}{1-2r}] \subseteq \Delta_v (C)$.

By assumption,  $\{S_I\}_I$ is uniformly dense with respect to $C$. Since thickness is preserved by translations (translating also the system of balls), we have
\[
\tau(C, \{S_I\}_I) \tau(C+tv, \{S_I+tv\}_I) \geq \frac{1}{(1-2r)^2}.
\]
It remains to show that $C \cap (1-2r) \left(B[0,1]+tv \right)\neq \emptyset$. Since $r \in (0, \frac{1}{3}]$, we have that $(1-2r) \left(B[0,1]+tv \right)$ is a ball of radius at least $r$.
And since $t \in [0, \frac{2r}{1-2r}]$, we have $(1-2r) \left(B[0,1]+tv \right) \subseteq B[0,1]$. Hence, by $r$-denseness, there is a child $S_i$ of $S_{\emptyset}$, contained in $(1-2r) \left(B[0,1]+tv \right)$. In particular, $C \cap (1-2r) \left(B[0,1]+tv \right)\neq \emptyset$.
\end{proof}

What is new compared to Mattila and Sjölin's result is that this holds in every direction, and we get a uniform explicit interval containing $0$. The assumptions however are much stronger.

\subsection{Patterns in thick sets in $\R^d$}

Recall from Theorem \ref{thm:y21} that sets of large Newhouse thickness $\tau$ contain homothetic copies of all finite sets of certain explicit cardinality $N(\tau)$, and that this result is based on showing that thick sets are winning for the potential game. Both the result and the approach generalize to both our definitions of thickness in $\R^d$. In the case of Definition \ref{def:thickness-Rd}, the proof uses a variant of the game in which Alice erases neighbourhoods of cube boundaries (spheres in the $d_\infty$ metric) instead of balls. Since the Gap Lemma is not used in the arguments, no denseness assumption is needed for these results. See \cite[Theorem 7]{FalconerYavicoli} and \cite[Theorem 20]{y22} for details.



\begin{thebibliography}{10}

\bibitem{Astels}
Stephen Astels.
\newblock Cantor sets and numbers with restricted partial quotients.
\newblock {\em Trans. Amer. Math. Soc. 352 (2000), no. 1, 133--170.}


\bibitem{Biebler}
Sébastien Biebler.
\newblock A complex gap lemma.
\newblock {\em Proc. Amer. Math. Soc. 148 (2020), no. 1, 351--364.}


\bibitem{BP22}
Zack Boone and Eyvindur Ari Palsson.
\newblock A pinned {M}attila-{S}jölin type theorem for product sets.
\newblock Preprint,  arXiv:2210.00675.

\bibitem{orponen}
Tuomas Sahlsten, Borys Kuca and Tuomas Orponen.
\newblock On a continuous S´ark¨ozy type problem.
\newblock Preprint, arXiv:2110.15065.

\bibitem{BFS}
Ryan Broderick, Lior Fishman and David Simmons.
\newblock Quantitative results using variants of Schmidt's game: dimension bounds, arithmetic progressions, and more.
\newblock {\em Acta Arith. 188 (2019), no. 3, 289--316.}

\bibitem{CLP}
Vincent Chan, Izabella Łaba and Malabika Pramanik.
\newblock Finite configurations in sparse sets.
\newblock {\em J. Anal. Math. 128 (2016), 289--335.}

\bibitem{falconer_1985}
Kenneth Falconer
\newblock The geometry of fractal sets.
\newblock {\em Cambridge Tracts in Mathematics, 85. Cambridge University Press, Cambridge, 1986.}

\bibitem{FalconerTechniques}
Kenneth Falconer.
\newblock Techniques in fractal geometry.
\newblock {\em John Wiley \& Sons, Ltd., Chichester, 1997.}

\bibitem{FalconerFG}
Kenneth Falconer.
\newblock Fractal geometry. Mathematical foundations and applications. Third edition.
\newblock {\em John Wiley \& Sons, Ltd., Chichester, 2014.}

\bibitem{FalconerYavicoli}
Kenneth Falconer and Alexia Yavicoli.
\newblock Intersections of thick compact sets in $\mathbb{R}^d$.
\newblock {\em Math. Z.,301(3):2291–2315, 2022.}

\bibitem{DJF}
De-Jun Feng and Yu-Feng Wu.
\newblock On arithmetic sums of fractal sets in $\mathbb{R}^d$.
\newblock {\em J. Lond. Math. Soc. (2),104(1):35–65, 2021.}

\bibitem{HLP}
Kevin Henriot, Izabella Łaba, and Malabika Pramanik.
\newblock On polynomial configurations in fractal sets.
\newblock {\em Anal. PDE, 9(5):1153–1184, 2016.}

\bibitem{HKY}
Brian Hunt, Ittai Kan and James Yorke.
\newblock When Cantor sets intersect thickly.
\newblock {\em Trans. Amer. Math. Soc., 339(2):869–888, 1993.}

\bibitem{Kel99}
Tamas Keleti.
\newblock A 1-dimensional subset of the reals that intersects each of its translates in at most a single point.
\newblock {\em Real Anal. Exchange, 24(2):843–844, 1998/99.}

\bibitem{Kel08}
Tamas Keleti.
\newblock Construction of one-dimensional subsets of the reals not containing similar copies of given patterns.
\newblock {\em Anal. PDE, 1(1):29–33, 2008.}

\bibitem{LP09}
Izabella Łaba and Malabika Pramanik.
\newblock Arithmetic progressions in sets of fractional dimension.
\newblock {\em Geom. Funct. Anal., 19(2):429–456, 2009.}

\bibitem{Maga}
Peter Maga.
\newblock Full dimensional sets without given patterns.
\newblock {\em Real Anal. Exchange, 36(1):79–90,2010/11.}

\bibitem{Mathe}
Andras Mathe.
\newblock Sets of large dimension not containing polynomial configurations.
\newblock {\em Adv. Math.,316:691–709, 2017.}

\bibitem{MattilaSjolin}
Pertti Mattila and Per Sjolin.
\newblock Regularity of distance measures and sets.
\newblock {\em Math. Nachr., 204:157–162, 1999.}

\bibitem{MT21}
Alex McDonald and Krystal Taylor.
\newblock Finite point configurations in products of thick Cantor sets and a robust nonlinear Newhouse gap lemma.
\newblock Preprint, arXiv:2111.09393, 2021.

\bibitem{Newhouse}
Sheldon Newhouse.
\newblock Nondensity of axiom $A(a)$ on $S^2$.
\newblock {\em In Global Analysis (Proc. Sympos. Pure Math., Vol. XIV, Berkeley, Calif., 1968), pages 191–202. Amer. Math. Soc., Providence, R.I.,
1970.}

\bibitem{Newhouse2}
Sheldon Newhouse.
\newblock The abundance of wild hyperbolic sets and nonsmooth stable sets for diffeomorphisms.
\newblock {\em Inst. Hautes ´Etudes Sci. Publ. Math., (50):101–151, 1979.}

\bibitem{PalisTakens}
Jacob Palis and Floris Takens.
\newblock Hyperbolicity and sensitive chaotic dynamics at homoclinic bifurcations
\newblock {\em volume 35 of Cambridge Studies in Advanced Mathematics. Cambridge University Press, Cambridge, 1993. Fractal dimensions and infinitely many attractors.}

\bibitem{ST20}
Karoly Simon and Krystal Taylor.
\newblock Interior of sums of planar sets and curves.
\newblock {\em Math. Proc. Cambridge Philos. Soc., 168(1):119–148, 2020.}

\bibitem{Williams}
Williams R. F.
\newblock How big is the intersection of two thick Cantor sets?
\newblock {\em In Continuum theory and dynamical systems (Arcata, CA, 1989), volume 117 of Contemp. Math., pages 163–175. Amer. Math. Soc., Providence, RI, 1991.}

\bibitem{YavLinear}
Alexia Yavicoli.
\newblock Large sets avoiding linear patterns.
\newblock {\em Proc. Amer. Math. Soc., 149(10):4057–4066,2021.}

\bibitem{Y21}
Alexia Yavicoli.
\newblock Patterns in thick compact sets.
\newblock {\em Israel J. Math., 244(1):95–126, 2021.}

\bibitem{y22}
Alexia Yavicoli.
\newblock Thickness and a gap lemma in $\mathbb{R}^d$.
\newblock To appear in International Mathematics Research Notices (IMRN), Preprint arXiv:2204.08428, 2022.

\bibitem{Y20}
Han Yu.
\newblock Fractal projections with an application in number theory.
\newblock {\em Ergodic Theory Dynam. Systems, to appear, 2020. arXiv:2004.05924}
\end{thebibliography}

\end{document}